\documentclass[journal]{IEEEtran}
\usepackage{amsthm}

\usepackage{amsmath,amssymb,amsfonts}
\usepackage{xcolor}
\usepackage{graphicx}
\usepackage{subcaption}

\usepackage{hyperref}
\hypersetup{colorlinks=true, linkcolor=blue, filecolor=magenta,
urlcolor=cyan}
\usepackage{algorithm}
\usepackage{algpseudocode}

\newtheorem{theorem}{Theorem}[section]
\newtheorem{lemma}[theorem]{Lemma}
\newtheorem{corollary}[theorem]{Corollary}
\newtheorem{proposition}[theorem]{Proposition}
\newtheorem{assumption}[theorem]{Assumption}
\newtheorem{definition}[theorem]{Definition}
\newtheorem{remark}{Remark}[section]

\newtheorem{problem}{Problem}

\newcommand{\subscr}[2]{#1_{\textup{#2}}}

\newcommand{\rank}{\operatorname{rank}}

\newcommand{\real}{\ensuremath{\mathbb{R}}}
\newcommand{\diff}{\mathrm{d}}

\newcommand{\setdef}[2]{\{#1 \; | \; #2\}}

\newcommand{\bulletsym}{\hbox{$\bullet$}}
\newcommand{\bulletend}{\relax\ifmmode\else\unskip\hfill\fi\bulletsym}

\newcommand{\longthmtitle}[1]{\mbox{}{\emph{(#1):}}}

\newcommand{\oprocendsymbol}{\hbox{$\bullet$}} 
\newcommand{\oprocend}{\relax\ifmmode\else\unskip\hfill\fi\oprocendsymbol}

\begin{document}
\title{Accelerated Safe Gradient Flow}


\author{Xuting Gao, Sonia Mart\'{i}nez \thanks{X. Gao and S.
Mart\'{i}nez are with the Department of Mechanical and Aerospace
Engineering at the University of California, San Diego, La Jolla, CA
92092 USA. (e-mail: \{xug003, soniamd\}@ucsd.edu.)} }

\maketitle

\begin{abstract}
  This letter introduces an accelerated safe gradient flow (A-SGF) for
  strongly convex constrained optimization problems. The A-SGF filters
  a nominal second-order accelerated flow through a second-order
  control barrier function quadratic program. We establish conditions
  for strict feasibility and local Lipschitz continuity of the
  resulting vector field, prove forward invariance of a lifted safe
  set, and show that its unique equilibrium is the constrained
  minimizer. A variational Lyapunov analysis isolates the effect of
  the safety correction and gives explicit conditions on the barrier
  gains under which A-SGF retains the same convergence rate as the
  unconstrained Nesterov accelerated method. A nonlinear 2D example
  demonstrate anytime feasibility and faster convergence of the
  algorithm.
\end{abstract}

\section{Introduction}
The safe gradient flow (SGF)~\cite{AA-JC:24-tac} casts a constrained
optimization in terms of a safety-critical control problem. By
filtering the negative gradient through control barrier functions
(CBFs), the SGF keeps every trajectory feasible while steering it
toward a constrained local minimizer. This viewpoint is particularly
appealing when the optimization dynamics are interconnected with a
physical system, where intermediate infeasible solutions may be
inadmissible. Thus, the SGF framework has been extended to safe
feedback optimization~\cite{GD-PM-JC-WPMHH:25-cdc,GD-JC-WPMHH:25} and
applied successfully to safety-critical optimal
regulation~\cite{AC-YC-BC-JC-EDA:24-tcst}.

Acceleration, on the other hand, plays a central role in first-order
optimization. Momentum methods, originating with Polyak's heavy-ball
method~\cite{BTP:64}, and Nesterov's accelerated gradient
method~\cite{YEN:83}, can significantly improve convergence. Their
continuous-time interpretations clarify the acceleration
mechanism~\cite{WS-SB-EJC:16}, 
including its variational and Lyapunov
perspectives~\cite{AW-ACW-MIJ:16,ACW-BR-MIJ:21}. For an $L$-smooth and
$\mu$-strongly convex objective, Nesterov's accelerated method for
strongly convex functions (\texttt{NAG-SC}) exhibits the
characteristic $1/\sqrt{\kappa}$ convergence-rate dependence, where
$\kappa=L/\mu$, compared with $1/\kappa$ for standard gradient flow.

Safety-critical control further motivates the adaptation of
accelerated methods with anytime safety; however, this task is
nontrivial. Projection-based inertial dynamics can preserve
feasibility, but the resulting vector field is generally
nonsmooth~\cite{HA-FA:00,MM-MIJ:26}. Primal-dual accelerated flows
drive a feasibility residual to zero asymptotically and therefore do
not guarantee that the primal trajectory remains feasible at every
finite time~\cite{XZ-JL-JC:23}. More fundamentally, when constraints
become active, the constraint-handling mechanism modifies the nominal
accelerated dynamics. Consequently, the nominal accelerated
convergence rate is not automatically retained. Precisely, due to the
properties of the SGF, overcoming these issues by integrating the
\texttt{NAG-SC} with  SGFs seems viable.  

Motivated by these questions, we develop an \emph{accelerated safe
  gradient flow} for constrained, strongly convex optimization. The
proposed method takes \texttt{NAG-SC} as the nominal second-order flow
and filters its acceleration through a second-order CBF quadratic
program. The manuscript's contributions are fourfold. First, we
formulate A-SGF and establish conditions for strict feasibility of the
CBF-QP and local Lipschitz continuity of the resulting vector field,
thereby guaranteeing well-defined dynamics. Second, we prove forward
invariance of the lifted safe set and show that A-SGF introduces no
undesired equilibria. Third, through a Euclidean variational analysis,
we characterize how the CBF correction enters the dissipation
inequality and derive checkable conditions on the first- and
second-order CBF gains under which the nominal $1/\sqrt{\kappa}$
convergence rate is retained. Finally, numerical examples illustrate
the safety and acceleration properties of
A-SGF.\footnote{\emph{Notation:} 
The set $C^i(\mathcal{X}, \mathcal{Y})$ 
denotes all the $i$-th continuously differentiable functions mapping from
$\mathcal{X}$ to $\mathcal{Y}$, with $C^\infty(\mathcal{X},\mathcal{Y})$
denoting the smooth function set. For $f\in C^2(\real^n,\real)$, $\nabla f(x)$
and $\nabla^2 f(x)$ denote its gradient and Hessian at $x$, respectively.  For
$a\in\real$, define $[a]_+\:=\max\{a,0\}$. Given $b\in C^1(\real^n,\real)$ and a
vector field $v:\real^n\to\real^n$, the Lie derivative of $b$ along $v$ is $
L_vb(x):=\nabla b(x)^\intercal v(x)$, and $L_v^r b(x) := L_v(L_v^{r-1}b(x))$,
for $r > 0$. An extended class-$\mathcal K$ function is a continuous, strictly
increasing function $\alpha:\real\to\real$ s.t.~$\alpha(0)=0$ and
$\lim_{r\to\infty}\alpha(r)=\infty$; the set of such functions is denoted by
$\mathcal K^e$.  For a positive integer $k$, let $[k]:=\{1,\ldots,k\}$. 
The number of elements of $\mathcal{A}$ is denoted as $|\mathcal{A}|$.
For a vector $v\in \real^k$, $\| v \|$ denotes its $2$-norm and $\| v
\|_1$ denotes its $1$-norm.
}

\section{Preliminaries and Problem Formulation}
\label{sec:preliminaries}

This section briefly presents background theory on constrained
optimization, the safe gradient flow and barrier functions that we
need to present the problem formulation.

\subsection{Constrained Optimization and Optimality Conditions}
\label{subsec:constrained-opt}
We consider the constrained optimization problem
\begin{equation}
\label{eq:origin-opt}
    \begin{aligned}
        \min_{x \in \real^n} \;  J(x), \qquad
        \text{s.t.} \; h_i(x) \geq 0,
        \qquad i \in [k],
    \end{aligned}
\end{equation}
with objective $J:\real^n\to\real$ and constraints
$h_i:\real^n\to\real$, $i\in[k]$. The corresponding feasible set is $
\mathcal{C} := \setdef{x\in\real^n}{h_i(x)\geq 0,\;i\in[k]}$. For
$x\in\mathcal{C}$, let $ \mathcal{I}(x) := \setdef{i\in[k]} {h_i(x)=0}
$ denote the active constraint set. We impose the following
assumptions throughout the paper.

\begin{assumption}[Convexity]
    \label{assump:cvx}
    The objective $J \in C^2(\real^n, \real)$ is $\mu$-strongly convex
    and $L$-smooth. The constraints $h_i \in C^3(\real^n, \real)$ are
    concave and satisfy Slater's condition \footnote{\emph{Slater's
    condition} holds if  
    $\exists \bar{x}\in\real^n$ s.t.~$h_i(\bar x) > 0$, for $i \in [k]$.
    }.\oprocend
\end{assumption}

Under Assumption~\ref{assump:cvx}, $\mathcal{C}$ is a closed convex
set with nonempty interior, and~\eqref{eq:origin-opt} admits a unique
optimizer $x^\star$. Moreover, Slater's condition guarantees strong
duality and the existence of a KKT multiplier
$\lambda^\star\in\real_{\geq0}^k$ satisfying
\begin{equation*}
    \nabla J(x^\star)-
    \sum_{i=1}^k
    \lambda_i^\star\nabla h_i(x^\star)= 0,
    \; \;
    \lambda_i^\star h_i(x^\star)=0, \; i \in [k].
\end{equation*} 
Since the problem is convex, these KKT conditions are both necessary
and sufficient for global optimality.

\subsection{Accelerated Gradient Flow}
\label{subsec:accelerated-flow}

Accelerated methods admit a continuous-time interpretation as
second-order dissipative gradient systems
\cite{BTP:64,ACW-BR-MIJ:21,BS-SSD-MIJ-WJS:22}. For a $\mu$-strongly
convex and $L$-smooth objective, consider the nominal accelerated
dynamics
\begin{equation}
\label{eq:nominal-accelerated-flow}
    \ddot{x} +
    2\sqrt{\mu/L}\,\dot{x} +
    1/L\nabla J(x)
    = 0.
\end{equation}
These dynamics coincide with the continuous-time Heavy-Ball and
Nesterov's accelerated gradient methods for strongly convex objectives
(\texttt{NAG-SC})~\cite{ACW-BR-MIJ:21,BS-SSD-MIJ-WJS:22}, for
particular coefficients.  
Introducing the velocity state, \eqref{eq:nominal-accelerated-flow}
can be written as
\begin{equation}
\label{eq:nominal-state-space}
    \dot{x}=v,
    \qquad
    \dot{v}=\subscr{a}{nom}(x,v),
\end{equation}
where the nominal acceleration is $\subscr{a}{nom}(x,v) :=
-2\sqrt{\mu/L}\,v - 1/L\nabla J(x)$. The
flow~\eqref{eq:nominal-accelerated-flow} exhibits an accelerated
exponential convergence rate with
$1/\sqrt{\kappa}$~\cite{ACW-BR-MIJ:21}.

\subsection{Safe Gradient Flow and High-Order CBFs}
\label{subsec:sgf}
The SGF ~\cite{AA-JC:24-tac} combines the negative gradient flow of an
objective function with a CBF-based safety filter.  Thus, initially
feasible trajectories can remain feasible while converging to KKT
points of the constrained optimization.
For $\alpha>0$, define
\begin{align}
    \mathcal{G}_{\alpha}(x) :=
    \arg\min_{\xi\in\real^n}
    \quad&
    \frac{1}{2}
    \Vert \xi+1/L \nabla J(x)\Vert^2 
    \label{eq:sgf}
    \\
    \text{s.t.}\quad&
    \nabla h_i(x)^\intercal\xi
    \geq
    -\alpha h_i(x),
    \; \; i \in [k]. \notag
\end{align}

\begin{definition}[SGF]
\label{def:sgf}
The dynamical system $\dot x = \mathcal{G}_\alpha (x)$ with
$\mathcal{G}_{\alpha}$ defined in~\eqref{eq:sgf} is called the
\emph{safe gradient flow}.
\oprocend
\end{definition}

\begin{lemma} 
  [Lipschitzness of SGF~{ \cite{{AA-JC:24-tac}}}] If $\{h_i\}_{i=1}^k$
  satisfy the MFCQ \footnote{ The \emph{Mangasarian--Fromovitz
  constraint qualification} (MFCQ) holds at $x \in \mathcal{C}$ if
  there exists $d\in\real^n$ such that $\nabla h_i(x)^\intercal d >
  0$, for all $i\in \mathcal{I}(x)$. } and CRC \footnote{ The
  \emph{constant rank condition}~(CRC) holds at $x$ if there exists a
  neighborhood $\mathcal{U}$ of $x$ s.t., $\forall
  \mathcal{J}\subseteq\mathcal{I}(x)$, $\forall y \in \mathcal{U}$,
  $\rank[\nabla h_i(y)]_{i\in\mathcal{J}}$ is constant. } for all
  $x\in\mathcal{C}$, then ~\eqref{eq:sgf} is strictly feasible and
  $\mathcal{G}_\alpha(x)$ is locally Lipschitz, for all $x \in
  \mathcal{C}$.  Further, the CBF inequalities in~\eqref{eq:sgf}
  render $\mathcal{C}$ forward invariant.  \oprocend
\end{lemma}


To extend the SGF construction to the second-order
dynamics~\eqref{eq:nominal-accelerated-flow}, we recall the high-order
control barrier functions (HOCBFs)~\cite{WX-CB:21}. Consider the
control-affine system
\begin{equation}
\label{eq:control-affine-system}
    \dot{z}
    =
    f(z)+g(z)u,
\end{equation}
where $z\in\real^{n_z}$, $u\in\real^m$, $f:\real^{n_z}\to\real^{n_z}$,
and $g:\real^{n_z}\to\real^{n_z\times m}$ are locally Lipschitz. Let
$b:\real^{n_z} \to \real$ have \emph{relative degree} $r$
w.r.t.~\eqref{eq:control-affine-system}\footnote{ That is, $L_gL_f^j
b=0$ for $j=0,\ldots,r-2$ and $L_gL_f^{r-1}b\neq0$.}. Let $\alpha_j
\in \mathcal{K}^e$, $j \in [r]$. Starting from $\varphi_0(z):=b(z)$,
define $ \varphi_j := \dot{\varphi}_{j-1} + \alpha_j(\varphi_{j-1})$
and $\mathcal{C}_j := \setdef{z\in\mathcal{D}}
{\varphi_{j-1}(z)\geq0}$. 

\begin{theorem}[HOCBF and forward invariance~\cite{WX-CB:21}]
\label{thm:hocbf-invariance}
If there exist $\alpha_j\in\mathcal K^e$, $j\in[r]$ s.t. $ \sup_u
\varphi_r(z,u)\geq0 $ for $ z\in\bigcap_{j=1}^r\mathcal C_j$, then $b$
is a \emph{high-order control barrier function}~(HOCBF) for
\eqref{eq:control-affine-system}. Any locally Lipschitz
controller $u(z)$ satisfying $\varphi_r(z,u(z))\geq0$ on
$\cap_{j=1}^r\mathcal C_j$ renders $\cap_{j=1}^r\mathcal C_j$ forward
invariant.\oprocend
\end{theorem}

\subsection{Problem Formulation}
\label{subsec:problem-formulation}
We can now state the objective of this paper.
\begin{problem}
\label{prob:asgf}
For the constrained optimization~\eqref{eq:origin-opt}, construct a
second-order system whose nominal acceleration is given
by~\eqref{eq:nominal-state-space} and whose constraints are enforced
through relative-degree-$2$ CBFs. Establish conditions under which the
resulting dynamics are well defined and locally Lipschitz, admit no
undesired equilibria in the feasible set, and retain the accelerated
convergence-rate dependence of the nominal flow.
\oprocend
\end{problem}

\section{Accelerated Safe Gradient Flow}
\label{sec:ASGF}

In this section, we construct an Accelerated Safe Gradient
Flow~(A-SGF) by filtering the nominal
acceleration~\eqref{eq:nominal-state-space} through
relative-degree-$2$ CBF constraints. After this, we analyze its
well-posedness and characterize its equilibria.

For each constraint $h_i$, choose $c_{i,1}, c_{i,2} > 0$ as associated
linear gains for the first- and second-order CBF, respectively.
Starting from $\varphi_{i,0}(x):=h_i(x)$, the first HOCBF function is
\begin{equation}
\label{eq:FOCBF}
    \varphi_{i,1}(x,v)
    :=
    \dot{\varphi}_{i,0}
    +
    c_{i,1}\varphi_{i,0}
    =
    \nabla h_i(x)^\intercal v
    +
    c_{i,1}h_i(x).
\end{equation}
Differentiating~\eqref{eq:FOCBF} along $\dot{x}=v$, $\dot{v}=a$ gives
the inequality
\begin{align}
    \varphi_{i,2}(x,v;a)
    &=
    \nabla h_i(x)^\intercal a
    +
    v^\intercal\nabla^2h_i(x)v
    + \label{eq:SOCBF-inequal}\\
    & \quad (c_{i,1}+c_{i,2})
        \nabla h_i(x)^\intercal v
    +
    c_{i,1}c_{i,2}h_i(x)
    \geq0 . \notag
\end{align}
Thus, we define the safe acceleration as the minimum-norm modification
of $\subscr{a}{nom}$ satisfying all the second-order CBFs:
\begin{align}
\subscr{a}{CBF}(x,v)
:=
\arg\min_{a\in\real^n}
\quad&
    \frac{1}{2}
    \Vert
        a-\subscr{a}{nom}(x,v)
    \Vert^2
    \label{eq:SOCBF-QP}
\\
\text{s.t.}\quad&
    \varphi_{i,2}(x,v;a)\geq0,
    \quad i\in[k].
    \notag
\end{align}

\begin{definition}[A-SGF]
\label{def:ASGF}
The dynamical system
\begin{equation}
\label{eq:A-SGF}
    \dot{x}=v,
    \qquad
    \dot{v}=\subscr{a}{CBF}(x,v),
\end{equation}
where $\subscr{a}{CBF}$ is defined by the
SOCBF-QP~\eqref{eq:SOCBF-QP}, is called the \emph{accelerated safe
gradient flow (A-SGF)}.
\oprocend
\end{definition}

We next establish conditions under which this SOCBF-QP is strictly
feasible and the A-SGF vector field is locally Lipschitz.

Given the first-order CBFs\eqref{eq:FOCBF}, define the lifted safe set
\begin{equation}
\label{eq:lifted-safe-set}
    \mathcal{S}
    :=\hspace*{-0.1cm}
      \setdef{(x,v)\in\real^{2n}\hspace*{-0.1cm}}
    {
    \hspace*{-0.1cm}    h_i(x)\geq0,\; \varphi_{i,1}(x,v)\geq0, i\in [k]
    }.
\end{equation}
For $(x,v)\in\mathcal{S}$, define the active first-order CBF index set
$
    \mathcal{I}_1(x,v)
    :=
    \setdef{i\in[k]}
    {\varphi_{i,1}(x,v)=0}
$.
We first characterize the strict feasibility of the SOCBF-QP.

\begin{proposition}\longthmtitle{Strict feasibility of the SOCBF-QP}
\label{prop:asgf-feasibility}
Fix $(x,v)\in\mathcal{S}$, and suppose that there exists a
$q\in\real^n$ such that
\begin{subequations}
\label{eq:feas-cond}
    \begin{align}
        &\dot{\varphi}_{i,1}(x,v;q)>0,
        && i\in\mathcal{I}_1(x,v),
        \label{eq:feas-cond-1}
        \\
        &c_{i,2}>
        \frac{
            [-\dot{\varphi}_{i,1}(x,v;q)]_+
        }{
            \varphi_{i,1}(x,v)
        },
        && i\notin\mathcal{I}_1(x,v).
        \label{eq:feas-cond-2}
    \end{align}
\end{subequations}
Then, the SOCBF-QP~\eqref{eq:SOCBF-QP} is strictly feasible at $(x,v)$
and admits a unique optimizer $\subscr{a}{CBF}(x,v)$.
\end{proposition}

\begin{proof}
  Firstly, consider $i\in\mathcal{I}_1(x,v)$. By definition,
  $\varphi_{i,1}(x,v)=0$, hence, $ \varphi_{i,2}(x,v;q) =
  \dot{\varphi}_{i,1}(x,v;q)>0, $ by~\eqref{eq:feas-cond-1}.  Now consider
  $i\notin\mathcal{I}_1(x,v)$. Since $(x,v)\in\mathcal{S}$,
  $\varphi_{i,1}(x,v)>0$. By~\eqref{eq:feas-cond-2},
$
    c_{i,2}\varphi_{i,1}(x,v)
    >
    [-\dot{\varphi}_{i,1}(x,v;q)]_+ \ge -\dot{\varphi}_{i,1}(x,v;q),
$
and therefore,
$
    \varphi_{i,2}(x,v;q)
    =
    \dot{\varphi}_{i,1}(x,v;q)
    +
    c_{i,2}\varphi_{i,1}(x,v)
    >0.
$
In this way, $\varphi_{i,2}(x,v;q)>0$, for all $i$, and Slater's
condition holds for the constraints in~\eqref{eq:SOCBF-QP}. By
continuity, the same strict inequalities hold in a sufficiently small
neighborhood of $(x,v)$. Besides, \eqref{eq:SOCBF-QP} is strictly
convex. It guarantees the uniqueness of the solution.
\end{proof}

By Proposition~\ref{prop:asgf-feasibility}, the optimizer
of~\eqref{eq:SOCBF-QP} is well-defined. We may therefore define the
active SOCBF-QP index set
$
    \mathcal{A}(x,v)
    :=
    \setdef{i\in[k]}
    {
        \varphi_{i,2}
        \bigl(x,v;\subscr{a}{CBF}(x,v)\bigr)=0
    }.
$
The next result characterizes the regularity needed to guarantee the
existence and uniqueness of solutions to the A-SGF dynamics. 

\begin{theorem}[Lipschitzness of A-SGF]
\label{thm:asgf-wellposed}
Let $(x,v)\in\mathcal{S}$ satisfy the conditions of
Proposition~\ref{prop:asgf-feasibility}. Suppose that the CRC holds at
$x$ for $\mathcal{A}(x,v)$. Then $\subscr{a}{CBF}$ is locally
Lipschitz at $(x,v)$.
\end{theorem}

\begin{proof}
  Let $a^\star:=\subscr{a}{CBF}(x,v)$. By
  Proposition~\ref{prop:asgf-feasibility}, there exists a strictly
  feasible direction $q$ s.t.~
  $\varphi_{i,2}(x,v;q)>0$, for all $ i \in [k]$. For each
  $i\in\mathcal{A}(x,v)$ we have $\varphi_{i,2}(x,v;a^\star)=0$. Since
  the constraints in~\eqref{eq:SOCBF-QP} are affine in $a$, for
  $d:=q-a^\star$, 
  $
    \nabla h_i(x)^\intercal d
    = \varphi_{i,2}(x,v;q)
    - \varphi_{i,2}(x,v;a^\star)>0$ for $ i \in \mathcal{A}(x,v) 
  $. Hence, the MFCQ holds for the SOCBF-QP at $a^\star$.

Moreover, $\nabla_a\varphi_{i,2}(x,v;a) = \nabla h_i(x)$, so the CRC
for the SOCBF-QP~\eqref{eq:SOCBF-QP} is precisely the CRC stated in
this theorem w.r.t the index set $\mathcal{A}(x,v)$.
By~\cite[Theorem~3.6]{JL:95}, the MFCQ and the CRC together with
Assumption~\ref{assump:cvx} imply that the $\subscr{a}{CBF}(x,v)$ is
locally Lipschitz at $(x,v)$. 
\end{proof}

The next more conservative condition depends only on the geometry of
the original constraints and is easier to check.

\begin{corollary}
\label{cor:asgf-conservative}
Suppose that, for every $x\in\mathcal{C}$, there exists
$\widetilde{q}\in\real^n$ s.t. $ \nabla
h_i(x)^\intercal\widetilde{q}>0 $, for all $i\in [k]$. Assume that the
CRC holds for all $x \in \mathcal{C}$ in~\eqref{eq:origin-opt}.
Then, for any positive CBF gains $c_{i,1}$ and $c_{i,2}$, the SOCBF-QP~\eqref{eq:SOCBF-QP}
is strictly feasible for every $(x,v)\in\mathcal{S}$, and
$\subscr{a}{CBF}$ is locally Lipschitz on $\mathcal{S}$.
\end{corollary}

\begin{proof}
  Assume that $\widetilde{q}$ satisfies the first stated condition.
  Fix any $(x,v) \in \mathcal{S}$, and take any $c_{i,1}, c_{i,2}>0$.
  We can choose $q = \beta\widetilde{q}$ to satisfy
  $\dot{\varphi}_{i,1}(x,v;q) > 0$, for all $i\in[k]$. In fact, any
  finite scaling parameter $\beta > 0$ will make~\eqref{eq:feas-cond}
  hold with
  \begin{equation*}
    \displaystyle
    \beta > \max_{i \in [k]} 
    \frac{[-v^\intercal \nabla^2 h_i(x) v 
    - c_{i,1} \nabla h_i(x)^\intercal v]_+}
    {\nabla h_i(x) \widetilde{q}} ,
  \end{equation*}
  Besides, since the gradient of~\eqref{eq:SOCBF-inequal} w.r.t.~$a$
  is $\nabla h_i(x)$, the CRC for the original
  optimization~\eqref{eq:origin-opt} for all $x\in \mathcal{C}$
  implies that the CRC holds for SOCBF-QP~\eqref{eq:SOCBF-QP}.
  By Theorem~\ref{thm:asgf-wellposed}, $\subscr{a}{CBF}$ is locally
  Lipschitz for all $(x,v) \in \mathcal{S}$.
\end{proof}
\begin{remark}
  Assumption~\ref{assump:cvx} implies the MFCQ holds on $\mathcal{C}$,
  which implies the first condition of Corollary~\ref{cor:asgf-conservative}.
  However, this does not guarantee the CRC.  \oprocend
\end{remark}
Finally, the lifted set~\eqref{eq:lifted-safe-set} provides the safety
of A-SGF.

\begin{corollary}[Safety of A-SGF]
\label{cor:asgf-safety}
Suppose that $\subscr{a}{CBF}$ is well defined and locally Lipschitz
on $\mathcal{S}$, and let $(x(0),v(0))\in\mathcal{S}$. Then
$\mathcal{S}$ is forward invariant along the A-SGF~\eqref{eq:A-SGF}.
\end{corollary}

\begin{proof}
  The forward invariance of $\mathcal{S}$ can immediately be shown by
  Theorem~\ref{thm:hocbf-invariance}. According
  to~\eqref{eq:lifted-safe-set}, $h_i(x) \ge 0$, for every $i \in
  [k]$, hence, $x(t) \in \mathcal{C}$ for $t \ge 0$.
\end{proof}

CBF-based safety filters are known to potentially introduce undesired
equilibria. Thus, establishing that the A-SGF introduces no additional
equilibria is essential, particularly when the constrained optimizer
lies on the boundary of $\mathcal{C}$. 

\begin{theorem}
  \longthmtitle{Equilibrium characterization}
\label{thm:equilibrium-characterization}
Under Assumption~\ref{assump:cvx}, suppose that the
A-SGF~\eqref{eq:A-SGF} is well defined on $\mathcal{S}$. Then the
equilibrium set of~\eqref{eq:A-SGF} restricted to
$\mathcal{S}$ is a singleton $\{[x^\star, 0]^\intercal\}$, where
$x^\star$ is the unique optimizer of problem~\eqref{eq:origin-opt}.
Therefore, the A-SGF admits no undesired equilibria in
$\mathcal{S}$.
\end{theorem}

\begin{proof}
  Let $(\bar{x},\bar{v})\in\mathcal{S}$ be an equilibrium of
  \eqref{eq:A-SGF}. By definition,
  $\bar{v} = \subscr{a}{CBF}(\bar{x},0)=0$, and
  $\subscr{a}{nom}(\bar{x},0) = -1/L \nabla J(\bar{x})$. The
  SOCBF-QP~\eqref{eq:SOCBF-QP} evaluated at $(\bar{x},0)$ is
\begin{align*}
    \min_{a\in\real^n}\quad&
    \frac{1}{2}
    \Big\|
        a+\frac{1}{L}\nabla J(\bar{x})
    \Big\|^2
    \\
    \text{s.t.}\quad&
    \nabla h_i(\bar{x})^\intercal a
    +
    c_{i,1}c_{i,2}h_i(\bar{x})
    \geq0,
    \qquad i\in[k],
\end{align*}
with unique optimizer $a^\star=\subscr{a}{CBF}(\bar{x},0)=0$. By
strict feasibility, the KKT conditions hold. Thus, there exist
multipliers $\eta_i\geq0$, $i\in[k]$, such that
\begin{equation*}
        \frac{1}{L}\nabla J(\bar{x})
        - \sum_{i=1}^k \eta_i\nabla h_i(\bar{x})
        = 0, 
        \; \;
        \eta_i
        \big( c_{i,1}c_{i,2}h_i(\bar{x}) \big)
        = 0.
\end{equation*}
With $c_{i,1}, c_{i,2} >0$ it follows $\eta_i h_i(\bar{x})=0$ for each
$i$. Define $\bar{\lambda}_i:=L\eta_i \ge 0$, above equality becomes
\begin{equation} \label{eq:SOCBF-eq-2}
    \nabla J(\bar{x})
    - \sum_{i=1}^k \bar{\lambda}_i\nabla h_i(\bar{x})
    = 0,
    \qquad \bar{\lambda}_i h_i(\bar{x})=0.
\end{equation}
Together with $h(\bar{x}) \ge 0$, \eqref{eq:SOCBF-eq-2} coincides with
the KKT conditions of original problem~\eqref{eq:origin-opt}.  Under
Assumption~\ref{assump:cvx}, the optimizer of~\eqref{eq:origin-opt} 
is unique. Hence, $\bar{x} = x^\star$.
\end{proof}

\section{Acceleration-Preserving Design}
\label{sec:design-CBF}

In this section, we characterize how the CBF safety correction affects
the accelerated convergence mechanism and provide sufficient
conditions on the HOCBF gains under which the A-SGF retains the
nominal accelerated rate. Throughout this section, define
$s:=\sqrt{\mu/L}=1/\sqrt{\kappa},$ and use the shorthands vectors
$h(x) = [h_i(x)]_{i = 1}^k$, $\varphi_1(x,v) = [\varphi_{i,1}(x,v)]_{i
= 1}^k$ and $\varphi_2(x,v;a) = [\varphi_{i,2}(x,v;a)]_{i = 1}^k$.

\subsection{Variational Interpretation and First-Order CBF Gain}
\label{subsec:variational-acceleration}

Following the variational interpretation of accelerated
optimization~\cite{AW-ACW-MIJ:16}, consider the Euclidean Lagrangian
\begin{equation*}
    \mathcal{L}(x,v,t)
    :=
    e^{2st}
    \Big[
        \frac{1}{2}\|v\|^2
        -
        \frac{1}{L}J(x)
    \Big].
\end{equation*}
The corresponding unforced Euler--Lagrange equations are
\begin{equation*}
    \frac{\diff}{\diff t}
    \Big(
        \frac{\partial\mathcal{L}}{\partial v}
    \Big)
    -
    \frac{\partial\mathcal{L}}{\partial x}
    =
    e^{2st}
    \Big(
        \ddot{x}
        +
        2s\dot{x}
        +
        \frac{1}{L}\nabla J(x)
    \Big)
    =
    0,
\end{equation*}
which recover the nominal 
flow~\eqref{eq:nominal-accelerated-flow}. Given the optimizer
$x^\star$, 
\begin{equation}
\label{eq:accelerated-Lyapunov}
    W(x,v)
    :=
    \frac{1}{L}
    (
        J(x)-J(x^\star)
    )
    +
    \frac{1}{2}
    \|
        v+s(x-x^\star)
    \|^2 .
\end{equation}
Since the A-SGF remains feasible and $x^\star$ minimizes $J$ over
$\mathcal C$, $W$ is nonnegative along every feasible trajectory. In
\cite{AW-ACW-MIJ:16}, $W$ serves as the Lyapunov function for the
unconstrained \texttt{NAG-SC}. The analysis of the variation of $W$
along the A-SGF will reveal a simple design for the first-order CBF
gains.

\begin{lemma}\longthmtitle{Time derivative of $W$}
\label{lem:accelerated-dissipation}
Consider the dynamics
$
    \dot{x}=v;
$ $
    \dot{v}
    =
    \subscr{a}{nom}(x,v)+q.
$
Then
\begin{equation*}
    \dot{W}
    \leq
    -sW
    -
    \frac{s}{2}\|v\|^2
    +
    [
        v+s(x-x^\star)
    ]^\intercal q .
\end{equation*}
\end{lemma}

\begin{proof}
Let $\varepsilon:=x-x^\star$. Differentiating
\eqref{eq:accelerated-Lyapunov} gives
\begin{align*}
    \dot W
    &=
    \frac{1}{L}\nabla J(x)^\intercal v
    +
    (v+s\varepsilon)^\intercal
    \Big(
        -sv-\frac{1}{L}\nabla J(x)+q
    \Big).
\end{align*}
Adding $sW$ to both sides and using $s^2=\mu/L$ yields
\begin{align*}
    \dot{W} +sW
    ={}&
    -\frac{s}{2}\|v\|^2
    +
    (v+s\varepsilon)^\intercal q
    \\
    &+
    \frac{s}{L}
    \Big[
        J(x)-J(x^\star)
        -
        \nabla J(x)^\intercal\varepsilon
        +
        \frac{\mu}{2}\|\varepsilon\|^2
    \Big].
\end{align*}
The last term is nonpositive by the $\mu$-strong convexity.
\end{proof}

For A-SGF, define the acceleration correction
$
    \subscr{q}{CBF}(x,v)
    :=
    \subscr{a}{CBF}(x,v)
    -
    \subscr{a}{nom}(x,v).
    $ This leads to a classical-mechanics interpretation of the A-SGF,
    where the $\subscr{q}{CBF}$ acts as a generalized constraint force
    on the nominal variational dynamics. In fact, the A-SGF satisfies
    the forced Euler--Lagrange equations from the Lagrange--d'Alembert
    principle of mechanics~\cite{MZ-FB:04a},
\begin{equation*}
    \frac{\diff}{\diff t}
    \Big(
        \frac{\partial\mathcal{L}}{\partial v}
    \Big)
    -
    \frac{\partial\mathcal{L}}{\partial x}
    =
    e^{2st}\subscr{q}{CBF}(x,v).
\end{equation*}
Thus, the safety filter does not alter the nominal potential or
dissipation directly; rather, it applies a generalized safety force.

The previous inequality can also be interpreted as a dissipativity
equation for the storage function $W$ with the supply
$
    \subscr{\Delta}{CBF}(x,v)
    :=
    [
        v+s(x-x^\star)
    ]^\intercal
    \subscr{q}{CBF}(x,v).
$
Lemma~\ref{lem:accelerated-dissipation} gives
$
    \dot{W}
    \leq
    -sW
    -
    \frac{s}{2}\|v\|^2
    +
    \subscr{\Delta}{CBF}.
$
Therefore, the design of the HOCBF gains amounts to controlling the work done by
the generalized safety force.

Define $\nabla h(x) := [ \nabla h_i(x)]_{i = 1}^k\in\real^{n\times k}$. 
Let $\lambda\in\real_{\geq0}^k$ be a KKT multiplier of the
SOCBF-QP~\eqref{eq:SOCBF-QP}. By the KKT condition
\begin{equation}
\label{eq:q-cbf-kkt}
    \subscr{q}{CBF}
    =
    \nabla h(x)\lambda,
    \qquad
    \lambda_i
    \varphi_{i,2}
    (x,v;\subscr{a}{CBF})
    =
    0,
    \quad i\in[k].
\end{equation}

Using~\eqref{eq:FOCBF} and the concavity of $h_i$, we obtain
\begin{align}
    \subscr{\Delta}{CBF}
    &=
    \sum_{i=1}^k
    \lambda_i
    [
        \nabla h_i(x)^\intercal v
        +
        s\nabla h_i(x)^\intercal(x-x^\star)
    ]
    \notag\\
    &\leq
    \sum_{i=1}^k
    \lambda_i
    [
        \varphi_{i,1}
        +
        (s-c_{i,1})h_i(x)
        -
        sh_i(x^\star)
    ].
    \label{eq:Delta-CBF-c1-bound}
\end{align}
Since $h_i(x)\geq0$, $h_i(x^\star)\geq0$, and $\lambda_i\geq0$,
\eqref{eq:Delta-CBF-c1-bound}, this motivates the selection
$c_{i,1}\geq s$.  In fact, the rate-matched choice
\begin{equation}
\label{eq:c1-design}
    c_{i,1}=s,
    \qquad i\in[k],
\end{equation}
eliminates the terms in $h_i(x)$ exactly and gives
$
    \subscr{\Delta}{CBF}
    \leq
    \lambda^\intercal\varphi_1.
$
Hence, the first-order CBF relaxation is taken to match the
exponential scale of the nominal accelerated flow.

However, the SOCBF complementarity condition~\eqref{eq:q-cbf-kkt} acts
on $\varphi_{i,2}$ rather than on $\varphi_{i,1}$, and
$\lambda^\intercal\varphi_1$ does not vanish in general. The
second-order gain must therefore compensate for the mismatch. The
search for a gain that can guarantee both the convergence and
accelerated rate motivates the construction of a modified Lyapunov
function in the next section.

\subsection{Second-Order Gain Design}
\label{subsec:c2-corrected-energy}

Let $\lambda^\star\in\real_{\geq0}^k$ be an optimal KKT multiplier
of~\eqref{eq:origin-opt}. We start by defining the optimal
scaled-Lagrangian gap
\begin{equation*}
    F(x)
    :=
    \frac{1}{L}
    \Big[
        J(x)-J(x^\star)
        +
        \sum_{i=1}^k
        \lambda_i^\star
        \bigl(
            h_i(x^\star)-h_i(x)
        \bigr)
    \Big].
\end{equation*}
Denote its gradient as
$
    r(x)
    :=
    \nabla F(x)
    =
    \frac{1}{L}
    [
        \nabla J(x)-\nabla h(x)\lambda^\star
    ].
$
The KKT conditions of~\eqref{eq:origin-opt} give $r(x^\star)=0$.
Moreover, Assumption~\ref{assump:cvx} and the concavity of $h_i$ imply
\begin{equation}
\label{eq:F-strong-convexity}
    \nabla^2F(x)
    =
    \frac{1}{L}
    \Big[
        \nabla^2J(x)
        -
        \sum_{i=1}^k
        \lambda_i^\star\nabla^2h_i(x)
    \Big]
    \succeq
    s^2I.
\end{equation}
Consequently,
$
    F(x)
    \geq
    \frac{s^2}{2}
    \|x-x^\star\|^2.
$
We now define a Lyapunov function candidate 
\begin{equation}
\label{eq:shifted-energy}
    W_F(x,v)
    :=
    F(x)
    +
    \frac{1}{2}
    \|
        v+s(x-x^\star)
    \|^2.
\end{equation}

\begin{lemma}[Time derivative of $W_F$]
\label{lem:shifted-accelerated-dissipation}
By choosing $c_{i,1} =s$, for $i \in [k]$, the A-SGF satisfies
\begin{equation*}
    \dot{W}_F
    \leq
    -sW_F
    -
    \frac{s}{2}\|v\|^2
    +
    \Big(
        \lambda-\frac{\lambda^\star}{L}
    \Big)^\intercal
    \varphi_1.
  \end{equation*}
\end{lemma}

\begin{proof}
  By complementary slackness, $\lambda^{\star\intercal}h(x^\star)=0$,
  and by definition of $W$ and $W_F$ we have
  $ W_F = W - \frac{1}{L} \lambda^{\star\intercal}h(x)$.
  Moreover, $c_{i,1} =s$ gives $\varphi_1 = \nabla h(x)^\intercal v +
  sh$. Together with Lemma~\ref{lem:accelerated-dissipation}, the
  result follows.
\end{proof}

Unfortunately, the last term of the upper bound of $\dot{W}_F$ will
not dissipate automatically. This motivates a further tweaking of the
Lyapunov function to cancel that last term.

Before introducing the final Lyapunov function, we provide a mild
assumption that bounds the Hessian of $F$ and $h$.

\begin{assumption}\longthmtitle{Hessian on relevant compact set}
\label{assump:acc-curvature}
There exists a compact forward-invariant set
$\mathcal K\subseteq\mathcal S$, containing $(x(0),v(0))$, and
constants $M<\infty$ and $0 < \bar H_i<\infty$, s.t.
$
        \nabla^2F(x)
        \preceq
        MI 
$ and        
$
        -\nabla^2h_i(x)
        \preceq
        \bar H_i I
$ for $
        i\in[k],
$ 
$(x,v) \in \mathcal K$. By the $L$-smoothness of $J$, a valid
choice is $ M := 1+ \frac{1}{L} \sum_{i=1}^k \lambda_i^\star\bar H_i$, and $M=1$ for affine constraints. 
\oprocend
\end{assumption}

With the above assumption, $F$ is convex and $M$-smooth on the considered
region, which implies, for all $x \in \mathcal{K}$,
\begin{equation}
\label{eq:r-F-bound}
    \|r(x)\|^2
    \leq
    2M F(x).
\end{equation}

For simplicity, we take uniform gains of the form
\begin{equation}
    \label{eq:c2-common-choice}
    c_{i,2} = c_2, \qquad i\in[k].
\end{equation}
\begin{lemma}[Monotonicity w.r.t.~$c_2$]
\label{lem:q-cbf-c2-monotonicity}
Denote the $\subscr{q}{CBF}$ with chosen $c_2 > 0$ as $\subscr{q}{CBF}^{c_2}$.
For any fixed $(x,v)\in\mathcal S$ and $c_2\geq c_2^0$,
there is
$
    \|
        \subscr{q}{CBF}^{c_2}(x,v)
    \|
    \leq
    \|
        \subscr{q}{CBF}^{c_2^0}(x,v)
    \|.
$
\end{lemma}

\begin{proof}
Since $(x,v)\in\mathcal S$, we have
$\varphi_{i,1}(x,v)\geq0$ for every $i\in[k]$. Hence, for
$c_2'\geq c_2$,
$
    \varphi_{i,2}^{c_2'}
    =
    \varphi_{i,2}^{c_2}
    +
    (c_2'-c_2)\varphi_{i,1}
    \geq
    \varphi_{i,2}^{c_2}.
    $ Thus, every point feasible for the SOCBF-QP~\eqref{eq:SOCBF-QP}
    with gain $c_2$ remains feasible for the problem with gain $c_2'$.
    Since the SOCBF-QP minimizes $\frac12\|q\|^2$, enlarging feasible
    set cannot increase the minimum norm, proving
    the result.
\end{proof}

\begin{lemma}
\label{lem:uniform-multiplier-bound}
Suppose that there exists $c_2^0>0$ s.t. the condition of
Proposition~\ref{prop:asgf-feasibility} or
Corollary~\ref{cor:asgf-conservative} holds everywhere on a compact
set $\mathcal{K}^0 \subseteq \mathcal{S}$  with $c_{i,2}=c_2^0$. Then,
there exists a constant $\underline{\sigma}>0$, independent of $c_2 >
c_2^0$, s.t.
$   
    \|\nabla h(x)\nu\|
    \geq
    \underline{\sigma}\|\nu\| 
$,
for every $(x,v)\in\mathcal K^0$ and $\nu\in\real_{\geq0}^k$ satisfying
$
    \nu^\intercal
    \varphi_2^{c_2}
    \bigl(
        x,v;
        \subscr{a}{CBF}^{c_2}(x,v)
    \bigr)
    =0.
$
Moreover, if $\lambda(x,v)$ is any KKT multiplier of the SOCBF-QP with $c_2$,
then
$
    \|\lambda(x,v)\|
    \leq
    \frac{1}{\underline{\sigma}}
    \|\subscr{q}{CBF}^{c_2}(x,v)\|.
$
\end{lemma}

\begin{proof}
We first establish a uniform strict-feasibility margin for the SOCBF-QP with
$c_2=c_2^0$. For every $z=(x,v)\in\mathcal K^0$, strict feasibility at $z$
implies that there exists a correction $q_z\in\real^n$ such that $
\varphi_{i,2}^{c_2^0} \bigl( z; \subscr{a}{nom}(z)+q_z \bigr) >0$ for $i\in[k]$.
By continuity, the same $q_z$ remains strictly feasible in a neighborhood
$\mathcal U_z$ of $z$. Hence, there exists $\epsilon_z>0$ such that $
\varphi_{i,2}^{c_2^0} \bigl( z'; \subscr{a}{nom}(z')+q_z \bigr)\geq \epsilon_z$,
for $ z'\in\mathcal U_z$, $i\in[k]$. Since $\mathcal K^0$ is compact, a finite
collection $\{\mathcal U_{z_j}\}_{j=1}^N$ covers $\mathcal K^0$. Define
$\epsilon := \min_{j\in[N]}\epsilon_{z_j}>0$ and $Q := \max_{j\in[N], \; z_j \in
\mathcal{K}^0}\|q_{z_j}\|<\infty$. Thus, for every $z\in\mathcal K^0$, there
exists a $q_s(z)$ s.t. $ \|q_s(z)\|\leq Q $ and $ \varphi_{i,2}^{c_2^0} \bigl(
z; \subscr{a}{nom}(z)+q_s(z) \bigr) \geq\epsilon $ for $i \in [k]$.

Lemma~\ref{lem:q-cbf-c2-monotonicity} implies $\Vert\subscr{q}{CBF}^{c_2}\Vert
\le \Vert \subscr{q}{CBF}^{c_2^0} \Vert \le \Vert q_s \Vert \le Q$. Now fix any
$c_2\geq c_2^0$. From the proof of Lemma~\ref{lem:q-cbf-c2-monotonicity}, there
is $ \varphi_{i,2}^{c_2} \bigl( x,v; \subscr{a}{nom}+q_s \bigr) \geq \epsilon$.
Therefore, the same $q_s$ is uniformly strictly feasible for every $c_2\geq
c_2^0$. It follows that $\nabla h_i(x)^\intercal(q_s-\subscr{q}{CBF}^{c_2})\geq
\epsilon$, for $i\in\mathcal A^{c_2}(x,v)$, where the last set denotes the
active index set of SOCBF-QP with $c_2$. Let $\nu\geq0$ satisfy the
complementarity condition in the statement. Hence, $ \epsilon\|\nu\|_1 \leq
(q_s-\subscr{q}{CBF}^{c_2})^\intercal\nabla h(x)\nu \leq
\|q_s-\subscr{q}{CBF}^{c_2}\| \|\nabla h(x)\nu\|$. With the inequality $
\|q_s-\subscr{q}{CBF}^{c_2}\| \leq 2Q$ we have $ \|\nabla h(x)\nu\| \geq
\frac{\epsilon}{2Q}\|\nu\|_1 \geq \frac{\epsilon}{2Q}\|\nu\|$. Thus, one may
choose $ \underline{\sigma} := \frac{\epsilon}{2Q}>0 $. Finally, let
$\lambda(x,v)$ be a KKT multiplier of the SOCBF-QP then $ \subscr{q}{CBF}^{c_2}
= \nabla h(x)\lambda(x,v)$. The result follows.
\end{proof}
Denote $\subscr{q}{CBF} = \subscr{q}{CBF}^{c_2}$, $\varphi_2 =
\varphi_2^{c_2}$ with a fixed choice $c_{i,2} = c_2$ throughout the
rest of the paper. With constant $\bar{c}_h :=
\frac{Q}{\underline{\sigma}} \sum_{i \in [k]} H_i$,
Lemma~\ref{lem:uniform-multiplier-bound} and
Assumption~\ref{assump:acc-curvature}, there is 
$
- \sum_{i=1}^{k}
\lambda_i(x,v) v^\intercal \nabla^2 h_i v \le \bar{c}_h \| v \|^2
$.

Recall that $\lambda$ (resp.~$\lambda^\star$) is a KKT multiplier
of~\eqref{eq:SOCBF-QP} (resp.~of~\eqref{eq:origin-opt}). Then
$\lambda^\intercal \varphi_2 = 0$. Define the shorthands $ \widehat{q}
:= \subscr{q}{CBF} - \frac{1}{L} \nabla h(x)\lambda^\star $, $ \eta :=
\lambda - \frac{\lambda^\star}{L}$ and $ \chi := [ v^\intercal
\nabla^2 h_i v ]_{i=1}^k$. Then, $\widehat{q} = \nabla h(x)\eta$, and
the time derivative of $v$ along A-SGF
\begin{equation}
    \label{eq:shifted-ASGF}
    \dot v = -2sv - r(x) + \widehat{q}.
\end{equation}
Together with the choice of $c_{i,1}$ in~\eqref{eq:c1-design}, we have
$ \varphi_2 = \nabla h(x)^\intercal (\widehat{q} - r - sv) + \chi +
c_2 \varphi_1$.
Multiplying the previous equation by $\eta^\intercal$, using
$\widehat{q} = \nabla h(x) \eta$ yields
\begin{equation}
    \label{eq:c2-eta-varphi2}
    c_2 \eta^\intercal \varphi_1 =  (r + sv)^\intercal \widehat{q}
                                   -
                                   \| \widehat{q} \|^2
                                   -
                                   \eta^\intercal \chi
                                   -
                                   \frac{1}{L}(\lambda^\star)^\intercal 
                                   \varphi_2.
\end{equation}
Equation~\eqref{eq:c2-eta-varphi2} exposes the favorable quadratic
term $-\|\widehat{q}\|^2$, but also produces a new mixed term $(r +
sv)^\intercal \widehat{q}$.  At this point, we take an auxiliary term
to modify $W_F$. For $\gamma>0$, define
\begin{equation}
\label{eq:P-gamma}
    P_\gamma(x,v)
    :=
    -
    \frac{s}{2}\|v\|^2
    -
    r(x)^\intercal v
    -
    \gamma F(x),
\end{equation}
and define the following function
\begin{equation}
\label{eq:corrected-storage}
    \mathcal{W}(x,v)
    :=
    W_F(x,v)
    +
    \frac{1}{c_2}P_\gamma(x,v).
\end{equation}

The candidate $\mathcal{W}$ will serve as the Lyapunov function of the
A-SGF, by selecting the gain $c_2$ appropriately. This choice should
play two roles: first, to make $\dot{\mathcal{W}} \le -s\mathcal{W}$;
second, to make $\mathcal W$ positive-definite. Consequently, we pick
$\gamma$ as
\begin{equation}
\label{eq:gamma-design}
    \gamma>\frac{2M}{s},
    \qquad
    \delta
    :=
    s\gamma-2M>0,
\end{equation}
and define the following parameters
\begin{align}
    C_\gamma
    :=
    \frac{s^2}{2}
    +
    \bar c_h
    +
    \frac{
        M(\gamma-2s)^2
    }{
        2(s\gamma-2M)
    },
    \; \;
    B_\gamma
    :=
    \gamma+\frac{M}{s}+4s.
    \label{eq:B-gamma}
\end{align}

\begin{figure*}[!htbp]
 \centering
 \includegraphics[width=0.90\textwidth]{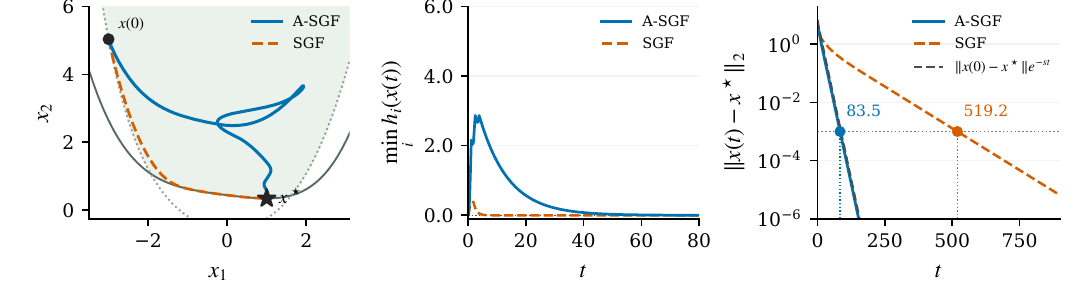}
 \caption{Nonlinear 2D example. (a) Trajectories of A-SGF and SGF. (b)
   Minimum constraints value. (c) Errors w.r.t.~time to show
   convergence; the marked times are when the error reaches and
   subsequently remains below $10^{-3}$.}
 \label{fig:2d-safety}
\end{figure*}

\begin{proposition}
  \longthmtitle{Time derivative of $\mathcal{W}$}
\label{prop:dot-mathcal-w}
Under Assumptions~\ref{assump:cvx} and \ref{assump:acc-curvature},
select gains of CBF uniformly as
in~\eqref{eq:c1-design},~\eqref{eq:c2-common-choice}, and choose
$\gamma$ by~\eqref{eq:gamma-design}. Then, it holds that
\begin{equation}
    \dot{\mathcal W}
    \leq
    -s\mathcal W
    -
    \Big(
        \frac{s}{2}
        -
        \frac{C_\gamma}{c_2}
    \Big)\|v\|^2
    -
    \frac{1}{c_2}\|\widehat q\|^2.
    \label{eq:dot-mathcal-w}
\end{equation}
\end{proposition}

\begin{proof}
Combining Lemma~\ref{lem:shifted-accelerated-dissipation}
and~\eqref{eq:c2-eta-varphi2} gives
\begin{align*}
    \dot{W}_F
    \leq{}&
    -sW_F
    -
    \frac{s}{2}\|v\|^2
    \\
    &+
    \frac{1}{c_2}
    \Big[
        (r+sv)^\intercal\widehat q
        -
        \|\widehat q\|^2
        -
        \eta^\intercal\chi
        -
        \frac{1}{L}
        \lambda^{\star\intercal}\varphi_2
    \Big].
\end{align*}
Since $\dot r=\nabla^2F(x)v$, differentiating \eqref{eq:P-gamma}
along~\eqref{eq:shifted-ASGF} gives
\begin{align*}
    \dot P_\gamma+sP_\gamma
    ={}&
    \frac{3s^2}{2}\|v\|^2
    +
    (2s-\gamma)r^\intercal v
    +
    \|r\|^2
    \\
    &-
    v^\intercal\nabla^2F(x)v
    -
    s\gamma F
    -
    (r+sv)^\intercal \widehat q.
\end{align*}
Hence, the terms $(r + sv)^\intercal \widehat{q}$ cancel exactly in
$\dot{\mathcal W}+s\mathcal W$. From concavity,
$\chi_i=v^\intercal\nabla^2h_i(x)v\leq0$. 
Thus, 
\begin{align*}
    -\eta^\intercal\chi
    =
    -\lambda^\intercal\chi
    +
    \frac{1}{L}
    \lambda^{\star\intercal}\chi
    \leq
    -\lambda^\intercal\chi
    \leq
    \bar c_h\|v\|^2,
\end{align*}
where the last inequality follows from
Assumption~\ref{assump:acc-curvature}, see the definition of
$\bar{c}_h$ after Lemma~\ref{lem:uniform-multiplier-bound}. Moreover,
$\lambda^{\star\intercal}\varphi_2\geq0$ and
\eqref{eq:F-strong-convexity} gives
$ v^\intercal\nabla^2F(x)v \geq
s^2\|v\|^2$. Using~\eqref{eq:r-F-bound} and~\eqref{eq:gamma-design},
\begin{equation*}
    \|r\|^2-s\gamma F
    \leq 2MF - s\gamma F 
    = -\delta F.
\end{equation*}
Finally, by Young's inequality, it holds that
\begin{align*}
    |(2s-\gamma)r^\intercal v|
    {} \le
    \delta F
    +
    \frac{
        M(\gamma-2s)^2
    }{
        2\delta
    }\|v\|^2.
\end{align*}
The substitution of these terms into $\dot{\mathcal{W}}$
yields~\eqref{eq:dot-mathcal-w}. 
\end{proof}

\begin{proposition}\longthmtitle{Positivity of $\mathcal{W}$}
\label{prop:positivity-mathcal-w}
Under the same conditions of Proposition~\ref{prop:dot-mathcal-w}, if
$
    c_2>B_\gamma,
$
then
\begin{equation}
\label{eq:corrected-storage-equivalence}
    \Big(
        1-\frac{B_\gamma}{c_2}
    \Big)
    W_F
    \leq
    \mathcal W
    \leq
    \Big(
        1+\frac{B_\gamma}{c_2}
    \Big)
    W_F.
\end{equation}
\end{proposition}

\begin{proof}
  From~\eqref{eq:shifted-energy}, $F(x) \leq W_F(x,v)$,
\begin{equation*}
    W_F
    \geq
    \frac{s^2}{2}\|x-x^\star\|^2
    +
    \frac{1}{2}
    \|
        v+s(x-x^\star)
    \|^2,
\end{equation*}
and
$
    \|v\|^2\leq 4 W_F.
$
    From Young's inequality
    and~\eqref{eq:r-F-bound},
\begin{align*}
    |r^\intercal v|
    \leq
    \frac{1}{2s}\|r\|^2
    +
    \frac{s}{2}\|v\|^2
    \leq
    \Big(
        \frac{M}{s}+2s
    \Big)W_F.
\end{align*}
It follows from~\eqref{eq:P-gamma} that
\begin{equation*}
    |P_\gamma|
    \leq
    \Big(
        \gamma+\frac{M}{s}+4s
    \Big)W_F
    =
    B_\gamma W_F.
\end{equation*}
Since $\mathcal W=W_F+P_\gamma/c_2$,
\eqref{eq:corrected-storage-equivalence} follows. 
\end{proof}
\begin{theorem}
\longthmtitle{CBF gain conditions for $1/\sqrt{\kappa}$ convergence}
\label{thm:acc-preserving-cbf-gains}
Pick $\gamma$ as in~\eqref{eq:gamma-design} and $c_{i,1} = s$. Let
Assumption~\ref{assump:cvx} and~\ref{assump:acc-curvature} hold.
Suppose that the A-SGF is well-defined on $\mathcal{S}$ with a $c_2^0
> B_\gamma > 0$.  Define the compact set $\mathcal{K}^0 =
\setdef{(x,v)\in\mathcal S} {\mathcal W^{c_2^0}(x,v)\le
(1+B_\gamma/c_2^0)W_F(x(0),v(0))}$. Let $\bar c_h^0$ and $C_\gamma^0$
be the corresponding bounds on $\mathcal K^0$. Then, any
$c_2>\max\{2C_\gamma^0/s,c_2^0\}$ makes $\mathcal W$
in~\eqref{eq:corrected-storage} a valid Lyapunov function for the
A-SGF and retains the nominal exponential-rate dependence
$s=1/\sqrt{\kappa}$ over $\mathcal{K}^0$.
\end{theorem}

\begin{proof}
Since $c_2^0>B_\gamma$, Proposition~\ref{prop:positivity-mathcal-w} guarantees
that $\mathcal W^{c_2^0}$ is positive definite and coercive, while
\eqref{eq:corrected-storage-equivalence} gives $\mathcal W^{c_2^0}(x(0),v(0))
\le(1+B_\gamma/c_2^0)W_F(x(0),v(0))$; hence $\mathcal K^0$ is compact and
contains $(x(0),v(0))$. 
For any $c_2\ge c_2^0$, Lemma~\ref{lem:q-cbf-c2-monotonicity} gives
$\bar c_h\le\bar c_h^0$ on $\mathcal K^0$, and hence $C_\gamma\le
C_\gamma^0$. Thus, for the stated choice of $c_2$,
Propositions~\ref{prop:dot-mathcal-w} and
\ref{prop:positivity-mathcal-w} imply $\dot{\mathcal
W}^{c_2}\le-s\mathcal W^{c_2}$ as long as the trajectory remains in
$\mathcal K^0$. Since $\mathcal
W^{c_2}(x,v)=W_F(x,v)+P_\gamma(x,v)/c_2$ and $P_\gamma\le B_\gamma
W_F$, for $c_2\ge c_2^0$ we have $\mathcal W^{c_2^0}\le
[(1+B_\gamma/c_2^0)/(1+B_\gamma/c_2)]\mathcal W^{c_2}$. Therefore,
$\mathcal W^{c_2^0}(x(t),v(t)) \le
e^{-st}(1+B_\gamma/c_2^0)W_F(x(0),v(0))$. Hence the trajectory cannot
leave $\mathcal K^0$. Consequently, $\mathcal W^{c_2}(x(t),v(t)) \le
e^{-st}\mathcal W^{c_2}(x(0),v(0))$, proving the result.
\end{proof}

\section{Numerical Examples}
\label{sec:numerical-examples}
We now implement the A-SGF with the following numerical case.
Let $x = [x_1,x_2]\in\real^2$, $a=0.08+0.1e^{1/2}$. Consider the cost
and constraint functions
\begin{align*}
 J(x)&=\tfrac12(x_1-1)^2+50(x_2-0.01 + 0.2e^{1/2})^2,\\
 h_1(x)&=x_2-0.02x_1^4-0.2e^{x_1/2}+a(x_1-1),\\
 h_2(x)&=x_2-0.008(x_1+2)^4-0.3e^{-x_1}+1.
\end{align*}
The objective $J$ is $1$-strongly convex and $100$-smooth. Both
constraint functions are $C^\infty$ and concave. We set the initial
point on the boundary with $x(0) = [-3, -0.992 + 0.3e^3]^\intercal$.
The unique constrained optimizer is
$x^\star=[1,0.02+0.2e^{1/2}]^\intercal$. The parameter of A-SGF are
$c_{i,1} = 0.1 = \sqrt{1/100}$ and $c_{i,2} = 10$. The results are
shown in Figure~\ref{fig:2d-safety}. From the left subfigure we can
see A-SGF retains the convergence rate of $s = 0.1$.

\section{Conclusion}
This letter introduced the A-SGF for convex optimization, which
converges to the unique constrained optimizer while preserving anytime
feasibility. We established conditions for well-posedness and derived
explicit HOCBF gain conditions under which A-SGF retains the
convergence rate of the nominal accelerated flow.

\bibliographystyle{IEEEtran}
\bibliography{alias,arxiv_ref}

\end{document}